\documentclass[12pt]{amsart}

\usepackage{amsfonts}
\usepackage[english]{babel}
\usepackage{amsmath}
\usepackage{latexsym,enumerate}
\usepackage{color}
\usepackage{color}
\usepackage{graphicx}

\numberwithin{equation}{section}

\newtheorem{theorem}{Theorem}
\newtheorem{lemma}{Lemma}

\newtheorem{proposition}{Proposition}
\newtheorem{remark}{Remark}
\newtheorem{definition}{Definition}

\numberwithin{theorem}{section}
\numberwithin{corollary}{section}
\numberwithin{lemma}{section}
\numberwithin{definition}{section}
\numberwithin{proposition}{section}
\numberwithin{remark}{section}

\newcommand{\medint}{-\kern  -,375cm\int}

\begin{document}
\title[Weighted symmetrization results for a problemwith variable Robin
parameter]{ Weighted symmetrization results for a problem with variable Robin
parameter}
\author[A. Alvino]{ Angelo Alvino }
\address{ Cristina Trombetti \\
Universit\`a degli studi di Napoli ``Federico II''\\
Dipartimento di Ma\-te\-ma\-ti\-ca e Applicazioni ``R. Caccioppoli''\\
Complesso di Monte Sant'Angelo, Via Cintia, 80126 Napoli, Italia. }
\email{angelo.alvino@unina.it}
\author[F. Chiacchio]{ Francesco Chiacchio }
\address{ Francesco Chiacchio \\
Universit\`a degli studi di Napoli ``Federico II''\\
Dipartimento di Ma\-te\-ma\-ti\-ca e Applicazioni ``R. Caccioppoli''\\
Complesso di Monte Sant'Angelo, Via Cintia, 80126 Napoli, Italia. }
\email{francesco.chiacchio@unina.it}
\author[C. Nitsch]{ Carlo Nitsch }
\address{ Carlo Nitsch \\
Universit\`a degli studi di Napoli ``Federico II''\\
Dipartimento di Ma\-te\-ma\-ti\-ca e Applicazioni ``R. Caccioppoli''\\
Complesso di Monte Sant'Angelo, Via Cintia, 80126 Napoli, Italia. }
\email{c.nitsch@unina.it}
\author[C. Trombetti]{ Cristina Trombetti }
\address{ Cristina Trombetti \\
Universit\`a degli studi di Napoli ``Federico II''\\
Dipartimento di Ma\-te\-ma\-ti\-ca e Applicazioni ``R. Caccioppoli''\\
Complesso di Monte Sant'Angelo, Via Cintia, 80126 Napoli, Italia. }
\email{cristina@unina.it}
\date{\today}
\maketitle

\begin{abstract}
By means of a suitable weighted rearrangement, we obtain various apriori
bounds for the solutions to a Robin problem. 
Among other things, we derive a family of Faber-Krahn type inequalities.
\end{abstract}

\vspace{.5cm}


\section{Introduction}

The last two decades have seen a growing interest in the study of the
Robin-Laplacian (see, e.g., \cite{B, BD, BG1, BG, D}). Recently, in \cite
{ANT}, it has been introduced a method that allows to obtain Talenti-type
results for this type of operator. The results in \cite{ANT} are quite
surprising since, as well known, the techniques introduced by Talenti in 
\cite{T} are tailored for problems whose solutions have level sets that do
not intersect the domain where the problem is defined, a phenomenon that
typically occurs when Robin boundary conditions are imposed.

Such an analysis is pushed even further in \cite{ACNT}, where the Robin
parameter is allowed to be a function bounded from above and below by two
positive constants. Let us emphasize that the methods used in \cite{ANT} and 
\cite{ACNT} are based on the classical Schwarz symmetrization, where, as
well known, the standard isoperimetric inequality plays a crucial role. For
related results see also \cite{K, CGNT} and the reference therein.

Here we consider the following problem

\begin{equation}
\left\{ 
\begin{array}{cc}
-\Delta u=f(x)\left\vert x\right\vert ^{l} & \text{in \ }\Omega \\ 
&  \\ 
\dfrac{\partial u}{\partial \nu }+\beta \left\vert x\right\vert ^{l/2  }u=0 & 
\text{on \ }\partial \Omega
\end{array}
\right.  \label{P}
\end{equation}
where, here and throughout the paper, $\Omega $ is a bounded Lipschitz
domain of $\mathbb{R}^{2}$ containing the origin,
$\beta >0,$ $l\in (-2,0],$ $\nu $ denotes the
outer unit normal to $\partial \Omega $ and $f(x)$ is a non negative
function in $L^{2}(\Omega ,\left\vert x\right\vert ^{l}dx)$, see sections 2
and 4 for definitions and some properties of this weighted space. A\ weak
solution to problem (\ref{P}) is a function $u\in H^{1}(\Omega )$ such that 
\begin{equation}
\int_{\Omega }\nabla u\nabla \psi dx+\beta \int_{\partial \Omega } u \psi
\left \vert x \right \vert ^{l/2} d \mathcal{H}^{1} 
=
 \int_{\Omega } f \psi
\left\vert x\right\vert ^{l}dx\text{ \ \ }\forall \psi \in H^{1}(\Omega ).
\label{Weak_Form}
\end{equation}
Our main results are based on a family of isoperimetric inequalities where
two different weights (which are powers of the distance from the origin)
appear in the perimeter and in the area element, respectively.

We will denote by $\Omega ^{\sharp }$ the disk centered at the origin of
radius $r^{\sharp }$, \ with $r^{\sharp }$ such that 
\begin{equation*}
\left\vert \Omega \right\vert _{l}:=\int_{\Omega }\left\vert x\right\vert
^{l}dx=\int_{\Omega ^{\sharp }}\left\vert x\right\vert ^{l}dx.
\end{equation*}
Let us introduce the so-called symmetrized problem 
\begin{equation}
\left\{ 
\begin{array}{cc}
-\Delta v=f^{\sharp }(x)\left\vert x\right\vert ^{l} & \text{in \ }\Omega
^{\sharp } \\ 
&  \\ 
\dfrac{\partial v}{\partial \nu }+\beta \left( r^{\sharp }\right) ^{l/2}v=0 & 
\text{on \ }\partial \Omega ^{\sharp },
\end{array}
\right.  \label{P_sharp}
\end{equation}
where $f^{\sharp }(x)$ is the unique radial and radially decreasing function
such that 
\begin{equation*}
\left\vert \left\{ x\in \Omega :f(x)>t\right\} \right\vert _{l}=\left\vert
\left\{ x\in \Omega :f^{\sharp }(x)>t\right\} \right\vert _{l}\text{ \ for
any }t\geq 0.
\end{equation*}
Our main results are contained in the following three theorems.

\vspace{.5cm}

\begin{theorem}
\label{L^1_L^2}Let $u$ and $v=v^{\sharp }$ be the solutions to problems (\ref
{P}) and (\ref{P_sharp}), respectively. If $f(x) \in L^{2}\left( \Omega
;\left\vert x\right\vert ^{l}dx\right) $ and $f(x)\geq 0$ a.e. in $\Omega $
then 
\begin{equation}
\left\Vert v\right\Vert _{L^{1}(\Omega ^{\sharp };\left\vert x\right\vert
^{l}dx)} = \int_{\Omega ^{\sharp }}\left\vert v(x)\right\vert \left\vert
x\right\vert ^{l}dx \geq \int_{\Omega }\left\vert u(x)\right\vert \left\vert
x\right\vert ^{l}dx = \left\Vert u\right\Vert _{L^{1}(\Omega ;\left\vert
x\right\vert ^{l}dx)}  \label{L^1}
\end{equation}
and 
\begin{equation}
\left\Vert v\right\Vert _{L^{2}(\Omega ^{\sharp };\left\vert x\right\vert
^{l}dx)}=\int_{\Omega ^{\sharp }}v^{2}(x)\left\vert x\right\vert ^{l}dx\geq
\int_{\Omega }u^{2}(x)\left\vert x\right\vert ^{l}dx=\left\Vert u\right\Vert
_{L^{2}(\Omega ;\left\vert x\right\vert ^{l}dx)}.  \label{L^2}
\end{equation}
\end{theorem}

\vspace{.5cm}

Theorem above can be significantly improved when the datum $f$ is constant.
More precisely the following holds true

\vspace{.5cm}

\begin{theorem}
\label{f(x)=1}If $f(x)=1$ a.e. in $\Omega $ then 
\begin{equation}
u^{\sharp }(x)\leq v(x)\text{ \ }x\in \Omega ^{\sharp }.  \label{f=1}
\end{equation}
\end{theorem}

\vspace{.5cm}

Let $\lambda _{1,l}(\Omega )$ and $\lambda _{1,l}(\Omega ^{\sharp })$ be the
first eigenvalues of the problems 
\begin{equation}
\left\{ 
\begin{array}{cc}
-\Delta u=\lambda (\Omega )\left\vert x\right\vert ^{l}u & \text{in \ }
\Omega  \\ 
&  \\ 
\dfrac{\partial u}{\partial \nu }+\beta \left\vert x\right\vert ^{l/2}u=0 & 
\text{on \ }\partial \Omega 
\end{array}
\right. 
\end{equation}
and 
\begin{equation}
\left\{ 
\begin{array}{cc}
-\Delta v
=
\lambda ( \Omega ^{\sharp } )
\left\vert x\right\vert ^{l}v & \text{in \ }
\Omega ^{\sharp } \\ 
&  \\ 
\dfrac{\partial u}{\partial \nu }+\beta \left( r^{\sharp }\right) ^{l/2}v=0 & 
\text{on \ }\partial \Omega ^{\sharp },
\end{array}
\right. 
\end{equation}
respectively. Then the following Faber-Krahn type inequality holds true.

\vspace{.5cm}

\begin{theorem}
\label{FK}It holds that 
\begin{equation}
\lambda _{1,l}(\Omega )\geq \lambda _{1,l}(\Omega ^{\sharp }).  \label{F_K}
\end{equation}
\end{theorem}

\vspace{.5cm}

\noindent The paper is organized as follows. In Section 2 we give the basic
definitions and results about the rearrangement with respect the measure $
\left\vert x\right\vert ^{l}dx$. We also recall the weighted isoperimetric
inequality on which such a rearrangement relies. Section 3 and 4 contain the
proofs of Theorems \ref{L^1_L^2} and \ref{f(x)=1}, respectively. The last
section is devoted to the proof of Theorem \ref{FK}. There we include
further comments on the relation between the spaces $H^{1}(\Omega )$ and $
L^{2}\left( \Omega ;\left\vert x\right\vert ^{l}dx\right) $. Note that, to
the best of our knowledge, it is the first time that a Faber-Krahn
inequality has been established when the Robin parameter is a function.

\begin{remark}
Since $0 \notin \partial\Omega$, our results still hold true if one assumes that the Robin parameter is a  function $ \beta(x) \in L^{\infty}(\partial \Omega)$ such that for some positive constant $C$ it holds 
$$
\beta(x) >C   \quad {\rm on } \quad \partial \Omega.
$$
Under this assumption the proofs are conceptually identical they just turn out to be  more cumbersome.
\end{remark}

\section{Preliminary results}

Let $\Omega$ be a Lebesgue measurable subset of $\mathbb{R}^{2}$ and let $l\in
(-2,0].$ Define 
\begin{equation*}
\left\vert \Omega\right\vert _{l}=\int_{\Omega}\left\vert x\right\vert ^{l}dx
\end{equation*}
and 
\begin{equation*}
P_{\frac{l}{2}}(\Omega)=\left\{ 
\begin{array}{cc}
\displaystyle\int_{\partial \Omega }\left\vert x\right\vert ^{\frac{l}{2}}d 
\mathcal{H}^{1} & \text{if }\Omega\text{ is }1-\text{rectifiable} \\ 
&  \\ 
+\infty & \text{otherwise.}
\end{array}
\right.
\end{equation*}
In the sequel with $D_{\rho }$ we will denote the disk centered at the
origin of radius $\rho $.

The following result is a particular case of a\ two-parameter family of
isoperimetric inequalities (see, e.g., \cite{ABCMP, ChiHo, DHHT} and the
references therein).

\vspace{0.5cm}

\begin{theorem}
It holds that 
\begin{equation*}
P_{\frac{l}{2}}(\Omega)\geq P_{\frac{l}{2}}(\Omega^{\sharp }),
\end{equation*}
where $\Omega^{\sharp }=D_{r^{\sharp }}$ with $r^{\sharp }>0:$ 
\begin{equation*}
\left\vert \Omega\right\vert _{l}=\left\vert \Omega^{\sharp }\right\vert _{l}.
\end{equation*}
\end{theorem}

\vspace{0.5cm}

\begin{remark}
Note that the isoperimetric inequality above can be written equivalently as
follows 
\begin{equation}
P_{\frac{l}{2}}^{2}(\Omega)\geq 2\pi \left( l+2\right) \left\vert \Omega\right\vert
_{l}.  \label{Is_Ineq}
\end{equation}
In fact an elementary computation shows that 
\begin{equation*}
P_{\frac{l}{2}}(\Omega^{\sharp })=2\pi \left( r^{\sharp }\right) ^{\frac{l+2}{2}}
\end{equation*}
and 
\begin{equation*}
\left\vert \Omega\right\vert _{l}=\left\vert \Omega^{\sharp }\right\vert
_{l}=\int_{\Omega ^{\sharp }}\left\vert x\right\vert ^{l}dx=\frac{2\pi }{l+2}
\left( r^{\sharp }\right) ^{l+2}
\end{equation*}
Recalling the definition of $r^{\sharp }$ we have 
\begin{equation*}
r^{\sharp }=\left[ \frac{l+2}{2\pi }\left\vert \Omega \right\vert _{l}\right]
^{\frac{1}{l+2}}.
\end{equation*}
Finally we deduce that 
\begin{equation*}
P_{\frac{l}{2}}(\Omega)\geq P_{\frac{l}{2}}(\Omega^{\sharp })=2\pi \left[ \frac{l+2}{
2\pi }\left\vert \Omega\right\vert _{l}\right] ^{\frac{1}{2}}
\end{equation*}
and, hence, the claim.
\end{remark}

\vspace{0.5cm}

Starting from this isoperimetric inequality one can consider the
corresponding weighted rearrangement of a function. For further reading on
this topic the reader can consult, for instance, \cite{Ka}, \cite{Ke}, \cite
{Ta2} and the references therein.

Let $u:\Omega \rightarrow \mathbb{R}$ be a measurable function.

\vspace{0.5cm}

\begin{definition}
The distribution function $\mu :t\in \left[ 0,\infty \right) \rightarrow 
\left[ 0,\infty \right) $ of $u$ is defined as 
\begin{equation*}
\mu (t)=\left\vert \left\{ x\in \Omega :\left\vert u(x)\right\vert
>t\right\} \right\vert _{l}.
\end{equation*}
\end{definition}

\begin{definition}
The decreasing rearrangement $u^{\ast }:s\in \left[ 0,\left\vert \Omega
\right\vert _{l}\right] \rightarrow \left[ 0,\infty \right] $ of $u$ is
defined as 
\begin{equation*}
u^{\ast }\left( s\right) =\inf \left\{ t\geq 0:\mu (t)<s\right\} .
\end{equation*}
\end{definition}

\begin{definition}
The weighted Schwarz symmetrization $u^{\sharp }(x):x\in \Omega ^{\sharp
}\rightarrow \left[ 0,\infty \right] $ of $u$ is defined as 
\begin{equation}
u^{\sharp }(x)=u^{\ast }\left( \left\vert D_{\left\vert x\right\vert
}\right\vert _{l}\right) =u^{\ast }\left( \frac{2\pi }{l+2}\left\vert
x\right\vert ^{l+2}\right) .  \label{u_star}
\end{equation}
\end{definition}

Equivalently one can say that $u^{\sharp }$ is the unique radial and
radially non-increasing function such that 
\begin{equation*}
\left\vert \left\{ x\in \Omega :\left\vert u(x)\right\vert >t\right\}
\right\vert _{l}=\left\vert \left\{ x\in \Omega :u^{\sharp }(x)>t\right\}
\right\vert _{l}\text{ \ for any }t\geq 0.
\end{equation*}

\vspace{0.5cm}

\begin{definition}
If $p\in \left[ 1,+\infty \right) $, we will denote by $L^{p}(\Omega
,\left\vert x\right\vert ^{l}dx)$ the space of all Lebesgue measurable real
valued functions $u$ such that 
\begin{equation}
\left\Vert u\right\Vert _{L^{p}(\Omega ,\left\vert x\right\vert ^{l}dx)} :=
\left( \int_{\Omega }\left\vert u\right\vert ^{p}\left\vert x\right\vert
^{l}dx\right) ^{1/p} <+\infty .  \label{Norm_Lp}
\end{equation}
\end{definition}

Note that since, by construction, $u$, $u^{\ast }$ and $u^{\sharp }$ are
equimeasurable we have that 
\begin{equation*}
\int_{\Omega }\left\vert u\right\vert ^{p}\left\vert x\right\vert ^{l}dx 
=
\int_{\Omega ^{\sharp }} \left( u^{\sharp }\right)^{p} \left\vert
x\right\vert ^{l}dx 
= \int_{0}^{\left\vert \Omega \right\vert _{l}} \left(
u^{\ast } \right)^{p} ds \text{ \ for any }p\geq 1.
\end{equation*}

We will need in the sequel the following well-known result (see, e.g., \cite
{CR}, \cite{Ka} and \cite{Ke}).

\vspace{0.5cm}

\begin{proposition}
Let $u\in L^{1}(\Omega ,\left\vert x\right\vert ^{l}dx)$ be non negative
function and let $E\subseteq \Omega $ be a measurable set. Then we have 
\begin{equation*}
\int_{E}u(x)\left\vert x\right\vert ^{l}dx \leq \int_{0}^{\left\vert
E\right\vert _{l}}u^{\ast }\left( s\right) ds.
\end{equation*}
\end{proposition}

We end this section by recalling the following version of Gronwall's Lemma

\vspace{0.5cm}

\begin{lemma}
\label{Gronwall}Let $\xi (\tau )$ be a continuously differentiable function,
satisfying, for some constant $C\geq 0$ the following differential
inequality 
\begin{equation*}
\tau \xi ^{\prime }(\tau )\leq \xi (\tau )+C\text{ \ for all }\tau \geq \tau
_{0}>0.
\end{equation*}
Then 
\begin{equation}
\xi (\tau )\leq \tau \frac{\xi (\tau _{0})+C}{\tau _{0}}-C\text{ \ for all }%
\tau \geq \tau _{0}\text{ }  \label{G1}
\end{equation}
and 
\begin{equation}
\xi ^{\prime }(\tau )\leq \frac{\xi (\tau _{0})+C}{\tau _{0}}\text{ \ for
all }\tau \geq \tau _{0}\text{.}  \label{G2}
\end{equation}
\end{lemma}

\vspace{0.5cm}

\section{The case $\ f(x)\in L^{2}(\Omega ,\left\vert x\right\vert ^{l}dx)$}

Let $u$ and $v$ the solutions to problems (\ref{P}) and (\ref{P_sharp}),
respectively. In the sequel the following notation will be in force.

For $t\geq 0$ we denote by 
\begin{equation}
U_{t}=\{x\in \Omega \colon u(x)>t\},\;\partial U_{t}^{\text{int}}=\partial
U_{t}\cap \Omega ,\;\partial U_{t}^{\text{ext}}=\partial U_{t}\cap \partial
\Omega ,  \label{Not_U}
\end{equation}
and by 
\begin{equation}
\mu (t)=|U_{t}|_{l}\,\,\text{and }\,P_{u}(t)=P_{\frac{l}{2}}(U_{t}).
\label{mu}
\end{equation}
Analogously if $t\geq 0$ we denote by 
\begin{equation}
V_{t}=\{x\in \Omega ^{\sharp }\colon v(x)>t\},\,\,\phi (t)=|V_{t}|_{l}\,\,
\text{ and }\,\,P_{v}(t)=P_{\frac{l}{2}}(V_{t}).
\end{equation}
The proof of our main theorems requires several auxiliary results, that may
have an interest of their own.


\vspace{0.5cm}

\begin{lemma}
\label{u_m<v_m} 
The following inequalities hold true 
\begin{equation}
0\leq u_{m}\leq v_{m},  \label{min}
\end{equation}
where 
\begin{equation*}
u_{m}:=\inf_{\overline{\Omega }}\,\,u,\text{ \ }v_{m}:=\min_{\overline{
\Omega ^{^{\#}}}}\,\,v.
\end{equation*}
\begin{equation}
u_{m}:=\inf_{\Omega }\,\,u\geq 0.  \label{min}
\end{equation}
\end{lemma}

\begin{proof}
Using $u^{-}=\max \left\{ 0,-u\right\} $ as test function in (\ref{Weak_Form}), we
obtain 
\begin{equation*}
0\geq -\int_{\Omega }\left\vert \nabla u^{-}\right\vert
^{2}dx-\int_{\partial \Omega }\left( u^{-}\right) ^{2}\left\vert
x\right\vert ^{l/2}d\mathcal{H}^{1}=\int_{\Omega }\left( u^{-}\right)
\left\vert x\right\vert ^{l}f(x)dx.
\end{equation*}
Hence $u^{-}=0$ a.e. in $\Omega ,$ and the first  inequality in (\ref{min}) is
verified. 
 We observe that the function $v(x)$  is radial and, therefore,  $v \equiv v_{m}$ on $\partial \Omega ^{\#}$. 
From equations \eqref{P} and \eqref{P_sharp} we easily deduce that
\begin{equation*}
\begin{split}
v_{m}\,P_{\frac{l}{2}}(\Omega ^{\sharp })& 
=
\int_{\partial \Omega ^{\sharp }}v(x)\left\vert x\right\vert ^{l/2}\,
d\mathcal{H}^{1} 
=-\frac{1}{\beta }
\int_{\partial \Omega ^{\sharp }}\frac{\partial v}{\partial \nu }\,
d\mathcal{H}^{1} 
\\
&
 =-\frac{1}{\beta }\int_{\Omega ^{\sharp }}\Delta v\,dx
=
\frac{1}{\beta }
\int_{\Omega ^{\sharp }} f^{\sharp }
\left\vert
x\right\vert ^{l}dx 
=
\frac{1}{\beta }
\int_{\Omega } f
\left\vert
x\right\vert ^{l}dx  \\
& 
=-\frac{1}{\beta }\int_{\Omega }\Delta u\,dx
=-\frac{1}{\beta }
\int_{\partial \Omega }\frac{\partial u}{\partial \nu }\,d\mathcal{H}^{1}
=
\int_{\partial \Omega }u\,\left\vert x\right\vert ^{l/2}
d\mathcal{H}^{1}
\geq
u_{m}P_{\frac{l}{2}}(\Omega )\geq u_{m}P_{\frac{l}{2} }(\Omega ^{\#}),
\end{split}
\end{equation*}
where in last inequality we have used the weighted isoperimetric inequality.
The claim is hence proven.

\end{proof}

\vspace{0.5cm}

\begin{lemma}
\label{Lemma_Dopo_Fubini}It holds 
\begin{equation*}
\int_{0}^{t}\tau \left( \int_{\partial U_{t}^{\text{ext}}}\dfrac{\left\vert
x\right\vert ^{l/2}}{u(x)}d\mathcal{H}^{1}\right) d\tau \leq \frac{
\displaystyle\int_{0}^{|\Omega |_{l}}f^{\ast }(s)ds}{2\beta }.
\end{equation*}
\end{lemma}

\begin{proof}
Fubini's Theorem gives 
\begin{gather*}
\int_{0}^{t}\tau \left( \int_{\partial U_{t}^{\text{ext}}}
\dfrac{\left\vert
x\right\vert ^{l/2}}{u(x)}d\mathcal{H}^{1}\right) d\tau 
\leq 
\int_{0}^{\infty
}\tau \left( \int_{\partial U_{t}^{\text{ext}}}
\dfrac{\left\vert
x\right\vert ^{l/2}}{u(x)}d\mathcal{H}^{1}\right) d\tau 
=
\int_{\partial \Omega
}\left( \int_{0}^{u(x)}\tau d\tau \right) \dfrac{\left\vert x\right\vert ^{l/2}
}{u(x)}d\mathcal{H}^{1} \\
=
\frac{1}{2}\int_{\partial \Omega }u(x)\left\vert x\right\vert ^{l/2}d\mathcal{
H}^{1}
=
\frac{1}{2\beta }\int_{\Omega }f(x)\left\vert x\right\vert ^{l}dx
=
\frac{\displaystyle\int_{0}^{|\Omega |_{l}}f^{\ast }(s)ds}{2\beta }.
\end{gather*}
\end{proof}

\vspace{0.5cm}

Now in order to render the notations less heavy we introduce two constants
 that will appear often in the following
\begin{equation}
\label{C(l)}
C(l)=2\pi \left( l+2\right) \text{ \ and \ }C(\Omega )=\frac{1}{2\beta ^{2}}
\left( \int_{0}^{\left\vert \Omega \right\vert _{l}}f^{\ast }(s)ds\right)
^{2}.
\end{equation}

\vspace{0.5cm}

\begin{lemma}
For almost every $t>0$ it holds that 
\begin{equation}
C(l)\mu (t)
\leq 
\left( -\mu ^{\prime }(t)+\frac{1}{\beta }\int_{\partial
U_{t}^{\text{ext}}}\dfrac{\left\vert x\right\vert ^{l/2}}{u(x)}d\mathcal{H}
^{1}\right) \left( \int_{0}^{\mu (t)}f^{\ast }(s)ds\right) 
\label{Prima_Fubini}
\end{equation}
and 
\begin{equation}
C(l)\phi (t)
=
\left( -\phi ^{\prime }(t)+\frac{1}{\beta }\int_{\partial
V_{t}^{\text{ext}}}\dfrac{\left\vert x\right\vert ^{l/2}}{v(x)}d\mathcal{H}
^{1}\right) \left( \int_{0}^{\phi (t)}f^{\ast }(s)ds\right) .
\label{equality}
\end{equation}
\end{lemma}

\begin{proof}
Let $t,h>0.$ Using the following test functions in (\ref{P}) 
\begin{equation*}
\varphi _{h}(x)=\left\{ 
\begin{array}{ccc}
0 & \text{if} & 0<u \leq t \\ 
h & \text{if} & u>t+h \\ 
u-t & \text{if} & t<u \leq t+h
\end{array}
\right. 
\end{equation*}
we obtain 
\begin{eqnarray*}
&&\int_{U_{t}\backslash U_{t+h}}\left\vert \nabla u\right\vert ^{2}dx
+
\beta
h\int_{\partial U_{t+h}^{\text{ext}}}u\left\vert x\right\vert ^{l/2}d\mathcal{H
}^{1}
+\beta  \int_{\partial U_{t}^{\text{ext}}\backslash \partial  U_{t+h}^{\text{ext }
}}u(u-t)\left\vert x\right\vert ^{l/2}d\mathcal{H}^{1} \\
&=&
\int_{U_{t}\backslash U_{t+h}}f(u-t)\left\vert x\right\vert
^{l}dx+h\int_{U_{t+h}}f\left\vert x\right\vert ^{l}dx.
\end{eqnarray*}
Dividing by $h$ and then letting $h$ go to $0$ in the previous equality and,
finally, using coarea formula, we obtain 
\begin{equation*}
\int_{\partial U_{t}}g(x) d \mathcal{H}^{1}
=
\int_{\partial U_{t}^{\text{int}}}\left\vert
\nabla u\right\vert d\mathcal{H}^{1}
+
\beta \int_{\partial U_{t}^{\text{ext}
}}u\left\vert x\right\vert ^{l/2}d\mathcal{H}^{1}=\int_{U_{t}}f\left\vert
x\right\vert ^{l}dx
\end{equation*}
where 
\begin{equation*}
g(x)=\left\{ 
\begin{array}{ccc}
\left\vert \nabla u\right\vert  & \text{if} & x\in \partial U_{t}^{\text{int}
} \\ 
&  &  \\ 
\beta u\left\vert x\right\vert ^{l/2} & \text{if} & x\in \partial U_{t}^{\text{
ext}}.
\end{array}
\right. 
\end{equation*}
On the other hand we have 
\begin{equation*}
P_{u}^{2}(t)\leq \left( \int_{\partial U_{t}}g(x)
d\mathcal{H}^{1}\right) \left( \int_{\partial U_{t}}\frac{\left\vert x\right\vert^{l}}{g(x)}
d \mathcal{H}^{1}\right) 
\end{equation*}
\begin{equation*}
=\left( \int_{\partial U_{t}}g(x)
d\mathcal{H}^{1}\right) 
\left( \int_{\partial U_{t}^{\text{int}}}\dfrac{\left\vert
x\right\vert ^{l}}{\left\vert \nabla u\right\vert }d\mathcal{H}^{1}
+
\frac{1}{\beta }\int_{\partial U_{t}^{\text{ext}}}\dfrac{\left\vert x\right\vert ^{l/2}
}{u(x)}d\mathcal{H}^{1}\right) 
\end{equation*}
\begin{equation*}
\leq \left( -\mu ^{\prime }(t)+\frac{1}{\beta }
\int_{\partial U_{t}^{\text{ext}}}\dfrac{\left\vert x\right\vert ^{l/2}}
{u(x)}d\mathcal{H}^{1}\right)
\left( \int_{0}^{\mu (t)}f^{\ast }(s)ds\right). 
\end{equation*}
Using the isoperimetric inequality (\ref{Is_Ineq}) we get the claim (\ref
{Prima_Fubini}).

Note that the distribution function $\phi $ of $v$ fulfills equality (\ref
{equality}), in place of inequality, since, as it is straightforward to
check, $v$ is a radial and radially decreasing function.
\end{proof}

Now we are in position to prove our first main result.

\vspace{0.5cm}

\begin{proof}[Proof of Theorem \protect\ref{L^1_L^2}.]
Multiplying both sides of (\ref{Prima_Fubini}) by $t$ and then integrating
over $(0,\tau )$, with $\tau \geq v_{m}$, we get 
\begin{eqnarray*}
&&C(l)\int_{0}^{\tau }t\mu (t)dt \\
&\leq &
\int_{0}^{\tau }\left( -t\mu ^{\prime }(t)\int_{0}^{\mu (t)}f^{\ast
}(s)ds\right) dt
+
\int_{0}^{\tau }\left( \frac{t}{\beta }\int_{\partial
U_{t}^{\text{ext}}}\dfrac{\left\vert x\right\vert ^{l/2}}{u(x)}d\mathcal{H}
^{1}\left( \int_{0}^{\mu (t)}f^{\ast }(s)ds\right) \right) dt \\
&\leq &
\int_{0}^{\tau }\left( -t\mu ^{\prime }(t)\int_{0}^{\mu (t)}f^{\ast
}(s)ds\right) dt
+
\frac{1}{\beta }\left( \int_{0}^{\left\vert \Omega
\right\vert _{l}}f^{\ast }(s)ds\right) \int_{0}^{\tau }\left(
t\int_{\partial U_{t}^{\text{ext}}}\dfrac{\left\vert x\right\vert ^{l/2}}{u(x)}
d\mathcal{H}^{1}\right) dt.
\end{eqnarray*}
Lemma \ref{Lemma_Dopo_Fubini} yields 
\begin{equation*}
C(l)\int_{0}^{\tau }t\mu (t)dt\leq \int_{0}^{\tau }-t\left( \int_{0}^{\mu
(t)}f^{\ast }(s)ds\right) \mu ^{\prime }(t)dt+\frac{1}{2\beta ^{2}}\left(
\int_{0}^{\left\vert \Omega \right\vert _{l}}f^{\ast }(s)ds\right) ^{2},
\end{equation*}
or equivalently 
\begin{equation}
C(l)\int_{0}^{\tau }t\mu (t)dt\leq \int_{0}^{\tau }-t\left[ \left(
\int_{0}^{\mu (t)}f^{\ast }(\sigma )d\sigma \right) \mu ^{\prime }(t)\right]
dt+C(\Omega ),  \label{Bef_parts}
\end{equation}
where $C(l)$ and $C(\Omega )$ are the constants defined in \eqref{C(l)}. An
integration by parts of the left hand side of (\ref{Bef_parts}) gives 
\begin{equation}
C(l)\int_{0}^{\tau }t\mu (t)dt=C(l)\left[ \tau \int_{0}^{\tau }\mu
(t)dt-\int_{0}^{\tau }\left( \int_{0}^{t}\mu (\sigma )d\sigma \right) dt
\right] .  \label{left}
\end{equation}
Setting 
\begin{equation*}
F(\omega ):=\int_{0}^{\omega }\left( \int_{0}^{\eta }f^{\ast }(s)ds\right)
d\eta ,
\end{equation*}
an integration by parts of the right hand side of (\ref{Bef_parts}) gives 
\begin{equation}
\int_{0}^{\tau }-t\left( \int_{0}^{\mu (t)}f^{\ast }(s)ds\right) d\mu
(t)=-\tau F(\mu (\tau ))+\int_{0}^{\tau }F(\mu (s))ds.  \label{right}
\end{equation}
From \eqref{left} and \eqref{right} we deduce that 
\begin{equation}
\tau F(\mu (\tau ))+C(l)\tau \int_{0}^{\tau }\mu (t)dt\leq \int_{0}^{\tau
}F(\mu (s))ds+C(l)\int_{0}^{\tau }\left( \int_{0}^{t}\mu (\sigma )d\sigma
\right) dt+C(\Omega ).  \label{beforeGronwall}
\end{equation}
Defining 
\begin{equation*}
\xi _{1}(\tau ):=\int_{0}^{\tau }F(\mu (s))ds+C(l)\int_{0}^{\tau }\left(
\int_{0}^{t}\mu (\sigma )d\sigma \right) dt
\end{equation*}
we can rewrite inequality \eqref{beforeGronwall} as follows 
\begin{equation*}
\tau \xi _{1}^{\prime }(\tau )\leq \xi _{1}(\tau )+C(\Omega ).
\end{equation*}
Gronwall Lemma (\ref{G2}), with $\tau _{0}=v_{m},$ gives 
\begin{equation}
F(\mu (\tau ))+C(l)\int_{0}^{\tau }\mu (t)dt\leq \widetilde{C},\text{ \ }
\tau \geq v_{m}.  \label{Ineq_mu}
\end{equation}
where 
\begin{equation*}
\widetilde{C}=\frac{\xi _{1}(v_{m})+C(\Omega )}{v_{m}}=\frac{\displaystyle
\int_{0}^{v_{m}}F(\mu (s))ds+C(l)\displaystyle\int_{0}^{v_{m}}\left( 
\displaystyle\int_{0}^{t}\mu (\sigma )d\sigma \right) dt+C(\Omega )}{v_{m}}.
\end{equation*}
While for $\phi (t),$ the distribution function of $v,$ we have the equality
sign 
\begin{equation}
F(\phi (\tau ))+C(l)\int_{0}^{\tau }\phi (t)dt=\widetilde{C}.  \label{Eq_phi}
\end{equation}
Inequalities (\ref{Ineq_mu}) and (\ref{Eq_phi}) clearly imply that 
\begin{equation*}
F(\phi (\tau ))+C(l)\int_{0}^{\tau }\phi (t)dt\geq F(\mu (\tau
))+C(l)\int_{0}^{\tau }\mu (t)dt,\text{ \ }\tau \geq v_{m}.
\end{equation*}
Since 
\begin{equation*}
\lim_{\tau \rightarrow +\infty }F(\phi (\tau ))=\lim_{\tau \rightarrow
+\infty }F(\mu (\tau ))=F(0)=0
\end{equation*}
we get 
\begin{equation*}
\left\Vert v\right\Vert _{L^{1}(\Omega ^{\sharp };\left\vert x\right\vert
^{l})}=\int_{0}^{\infty }\phi (t)dt\geq \int_{0}^{\infty }\mu
(t)dt=\left\Vert u\right\Vert _{L^{1}(\Omega ;\left\vert x\right\vert ^{l})},
\end{equation*}
i.e. our first claim, inequality (\ref{L^1}).

Now we want to establish the same comparison between the weighted $L^{2}$
-norms of $u$ and $v$, i.e. (\ref{L^2}). If $\tau $ goes to $+\infty $ in (
\ref{Bef_parts}) we obtain 
\begin{equation*}
C(l)\int_{0}^{\infty }t\mu (t)dt\leq \int_{0}^{\infty }F(\mu (\sigma
))d\sigma +C(\Omega ),
\end{equation*}
while for $\phi $ it holds 
\begin{equation*}
C(l)\int_{0}^{\infty }t\phi (t)dt=\int_{0}^{\infty }F(\phi (\sigma ))d\sigma
+C(\Omega ).
\end{equation*}
Therefore we get the claim once we show that 
\begin{equation}
\int_{0}^{\infty }F(\mu (\sigma ))d\sigma \leq \int_{0}^{\infty }F(\phi
(\sigma ))d\sigma .  \label{claim}
\end{equation}
Multiplying both sides of (\ref{Prima_Fubini}) by $\dfrac{tF(\mu (t))}{\mu
(t)}$ and then integrating over $(0,\tau )$ we get 
\begin{gather}
\label{AB}
A:=C(l)\int_{0}^{\tau }tF(\mu (t))d\tau \leq    \\
\int_{0}^{\tau }\left( -\mu ^{\prime }(t)
+
\frac{1}{\beta }\int_{\partial
U_{t}^{\text{ext}}}\dfrac{\left\vert x\right\vert ^{l/2}}{u(x)}d\mathcal{H}
^{1}\right) \dfrac{tF(\mu (t))}{\mu (t)}\left( \int_{0}^{\mu (t)}f^{\ast
}(s)ds\right) dt=:B  \notag
\end{gather}
Again an integration by parts gives 
\begin{eqnarray}
A &=&C(l)\left[ t\int_{0}^{s}F(\mu (t))d\tau \right] _{0}^{\tau
}-C(l)\int_{0}^{\tau }\left( \int_{0}^{s}F(\mu (\sigma ))d\sigma \right) 
\label{A} \\
&=&C(l)\tau \int_{0}^{\tau }F(\mu (\sigma ))d\sigma -C(l)\int_{0}^{\tau
}\left( \int_{0}^{s}F(\mu (\sigma ))d\sigma \right) .  \notag
\end{eqnarray}
Now set 
\begin{equation*}
H(\varrho ):=\int_{0}^{\varrho }\frac{F(\sigma )}{\sigma }\int_{0}^{\sigma
}f^{\ast }(s)ds
\end{equation*}
and 
\begin{eqnarray*}
B &:&=\int_{0}^{\tau }\left( -\mu ^{\prime }(t)\right) \dfrac{tF(\mu (t))}{
\mu (t)}\left( \int_{0}^{\mu (t)}f^{\ast }(s)ds\right) dt \\
&&
+
\int_{0}^{\tau }\left( \frac{1}{\beta }
\int_{\partial U_{t}^{\text{ext}}}
\dfrac{\left\vert x\right\vert ^{l/2}}{u(x)}d\mathcal{H}^{1}\right) \dfrac{
tF(\mu (t))}{\mu (t)}\left( \int_{0}^{\mu (t)}f^{\ast }(s)ds\right) dt.
\end{eqnarray*}
Furthermore define 
\begin{equation*}
B_{1}:=\int_{0}^{\tau }t\left[ \dfrac{F(\mu (t))}{\mu (t)}\left(
\int_{0}^{\mu (t)}f^{\ast }(s)ds\right) \left( -\mu ^{\prime }(t)\right) 
\right] dt=-\int_{0}^{\tau }t\left[ \frac{d}{dt}H(\mu (t))\right] dt
\end{equation*}
and 
\begin{equation*}
B_{2}:=\int_{0}^{\tau }\left( \frac{1}{\beta }\int_{\partial U_{t}^{\text{
ext }}}\dfrac{\left\vert x\right\vert ^{l/2}}{u(x)}d\mathcal{H}^{1}\right) 
\dfrac{tF(\mu (t))}{\mu (t)}\left( \int_{0}^{\mu (t)}f^{\ast }(s)ds\right)
dt,
\end{equation*}
so that 
\begin{equation*}
B=B_{1}+B_{2}.
\end{equation*}
Integrating by parts in $B_{1}$ we obtain 
\begin{equation}
B_{1}=-\left[ tH(\mu (t))\right] _{0}^{\tau }+\int_{0}^{\tau }H(\mu
(t))dt=-\tau H(\mu (\tau ))+\int_{0}^{\tau }H(\mu (t))dt.  \label{B1}
\end{equation}
Since, as it is easy to verify, $\dfrac{F(\rho )}{\rho }$ is a nondecreasing
function, using Lemma \ref{Lemma_Dopo_Fubini} we derive 
\begin{eqnarray}
B_{2} &\leq &
\frac{1}{\beta }\dfrac{F(\left\vert \Omega \right\vert _{l})}{
\left\vert \Omega \right\vert _{l}}\int_{0}^{\tau }\left( t\int_{\partial
U_{t}^{\text{ext}}}\dfrac{\left\vert x\right\vert ^{l/2}}{u(x)}d\mathcal{H}
^{1}\right) \left( \int_{0}^{\mu (t)}f^{\ast }(s)ds\right) dt  \label{B2} \\
&\leq &\frac{1}{2\beta ^{2}}\dfrac{F(\left\vert \Omega \right\vert _{l})}{
\left\vert \Omega \right\vert _{l}}\left( \int_{0}^{\left\vert \Omega
\right\vert _{l}}f^{\ast }(s)ds\right) ^{2}=:\widehat{C}  \notag
\end{eqnarray}
Collecting (\ref{AB}), (\ref{A}), (\ref{B1}) and (\ref{B2}) we infer that 
\begin{eqnarray}
&&\tau \left[ C(l)\int_{0}^{\tau }F(\mu (\sigma ))d\sigma +H(\mu (\tau ))
\right]   \label{csi} \\
&\leq &C(l)\int_{0}^{\tau }\left( \int_{0}^{s}F(\mu (\sigma ))d\sigma
\right) +\int_{0}^{\tau }H(\mu (t))dt+\widehat{C}.  \notag
\end{eqnarray}
Define 
\begin{equation*}
\xi _{2}(\tau )=C(l)\int_{0}^{\tau }\left( \int_{0}^{s}F(\mu (\sigma
))d\sigma \right) ds+\int_{0}^{\tau }H(\mu (t))dt,
\end{equation*}
then (\ref{csi}) can be rewritten as follows 
\begin{equation*}
\tau \xi _{2}^{\prime }(\tau )\leq \xi _{2}(\tau )+\widehat{C},\text{ \ }
\tau \geq v_{m}.
\end{equation*}
At this point Gronwall's Lemma (\ref{G2}) ensures that 
\begin{equation*}
\xi _{2}^{\prime }(\tau )\leq \frac{\xi _{2}(v_{m})+\widehat{C}}{v_{m}},
\text{ \ }\tau \geq v_{m}.
\end{equation*}
Therefore 
\begin{gather}
\xi _{2}^{\prime }(\tau )=C(l)\int_{0}^{\tau }F(\mu (\sigma ))d\sigma +H(\mu
(\tau ))  \label{Csi_prime_2} \\
\leq \frac{1}{v_{m}}\left[ C(l)\int_{0}^{v_{m}}\left( \int_{0}^{s}F(\mu
(\sigma ))d\sigma \right) ds+\int_{0}^{v_{m}}H(\mu (t))dt+\widehat{C}\right]
,\text{ \ }\tau \geq v_{m}.  \notag
\end{gather}
Since $F,H$ are nondecreasing functions and \ $\mu (\sigma )\leq \left\vert
\Omega \right\vert _{l}$ we have 
\begin{eqnarray*}
&&C(l)\int_{0}^{\tau }F(\mu (\sigma ))d\sigma +H(\mu (\tau )) \\
&\leq &\frac{1}{v_{m}}\left[ C(l)\int_{0}^{v_{m}}\left(
\int_{0}^{s}F(\left\vert \Omega \right\vert _{l})d\sigma \right)
ds+\int_{0}^{v_{m}}H(\left\vert \Omega \right\vert _{l})dt+\widehat{C}\right]
,\text{ \ }\tau \geq v_{m}.
\end{eqnarray*}
The last inequality can be equivalently written as follows 
\begin{eqnarray}
&&C(l)\int_{0}^{\tau }F(\mu (\sigma ))d\sigma +H(\mu (\tau ))
\label{Ineq_mu_2} \\
&\leq &C(l)F(\left\vert \Omega \right\vert _{l})\frac{v_{m}}{2}+H(\left\vert
\Omega \right\vert _{l})+\frac{\widehat{C}}{v_{m}}=:G(\left\vert \Omega
\right\vert _{l}),\text{\ \ }\tau \geq v_{m},  \notag
\end{eqnarray}
where 
\begin{equation*}
G(t):=C(l)F(t)\frac{v_{m}}{2}+H(t)+\frac{\widehat{C}}{v_{m}}.
\end{equation*}
As easy to verify, it holds that 
\begin{equation}
C(l)\int_{0}^{\tau }F(\phi (\sigma ))d\sigma +H(\phi (\tau ))=G(\left\vert
\Omega \right\vert _{l}).  \label{Eq_phi_2}
\end{equation}
Therefore from (\ref{Ineq_mu_2}) and (\ref{Eq_phi_2}) we get 
\begin{equation*}
C(l)\int_{0}^{\tau }F(\mu (\sigma ))d\sigma +H(\mu (\tau ))\leq
C(l)\int_{0}^{\tau }F(\phi (\sigma ))d\sigma +H(\phi (\tau ))
\end{equation*}
which implies 
\begin{equation*}
\lim_{\tau \rightarrow \infty }\left[ C(l)\int_{0}^{\tau }F(\mu (\sigma
))d\sigma +H(\mu (\tau ))\right] \leq \lim_{\tau \rightarrow \infty }\left[
C(l)\int_{0}^{\tau }F(\phi (\sigma ))d\sigma +H(\phi (\tau ))\right] .
\end{equation*}
Since 
\begin{equation*}
\lim_{\tau \rightarrow \infty }H(\mu (\tau ))=\lim_{\tau \rightarrow \infty
}H(\phi (\tau ))=0
\end{equation*}
we conclude that 
\begin{equation*}
\int_{0}^{\infty }F(\mu (\sigma ))d\sigma \leq \int_{0}^{\infty }F(\phi
(\sigma ))d\sigma ,
\end{equation*}
i.e. our claim (\ref{claim}).
\end{proof}

\section{The case $f(x)=1$}

As already pointed out when $f$ is constant the previous result can be
considerable sharpened. In this short Section we provide the proof of
Theorem (\ref{f(x)=1}). Here we will use the same notation of section 3.

Let $u$ and $v$ the solutions of the problems (\ref{P}) and (\ref{P_sharp})
with $f(x)=1$ a.e. in $\Omega ,$ that is 
\begin{equation}
\left\{ 
\begin{array}{cc}
-\Delta u=\left\vert x\right\vert ^{l} & \text{in \ }\Omega  \\ 
&  \\ 
\dfrac{\partial u}{\partial \nu }+\beta \left\vert x\right\vert ^{l/2}u=0 & 
\text{on \ }\partial \Omega 
\end{array}
\right.   \label{Eq}
\end{equation}
and 
\begin{equation}
\left\{ 
\begin{array}{cc}
-\Delta v=\left\vert x\right\vert ^{l} & \text{in \ }\Omega ^{\sharp } \\ 
&  \\ 
\dfrac{\partial v}{\partial \nu }+\beta \left( r^{\sharp }\right) ^{l/2}v=0 & 
\text{on \ }\partial \Omega ^{\sharp }.
\end{array}
\right.   \label{Eq_S}
\end{equation}

\vspace{0.5cm}

\begin{proof}[Proof of Theorem \protect\ref{f(x)=1}.]
Note, firstly, that when $f=1$ inequality (\ref{Prima_Fubini}) 
and equality (\ref{equality}) become 
\begin{equation}
C(l)\leq -\mu ^{\prime }(t)+\frac{1}{\beta }\int_{\partial U_{t}^{\text{ext}
}}\dfrac{\left\vert x\right\vert ^{l/2}}{u(x)}d\mathcal{H}^{1}
\label{before_I}
\end{equation}
and 
\begin{equation*}
C(l)=-\phi ^{\prime }(t)+\frac{1}{\beta }\int_{\partial U_{t}^{\text{ext}}} 
\dfrac{\left\vert x\right\vert ^{l/2}}{v(x)}d\mathcal{H}^{1},
\end{equation*}
respectively.
Multiplying (\ref{before_I}) by $t$ and then integrating over $(0,\tau )$,
with $\tau \geq v_{m}$, we get 
\begin{equation*}
2\pi \left( l+2\right) \frac{\tau ^{2}}{2}
\leq
 \int_{0}^{\tau }\left( -\mu
^{\prime }(\sigma )\sigma \right) d\sigma
 +
\frac{1}{\beta }\int_{0}^{\tau
}t\left( \int_{\partial U_{t}^{\text{ext}}}\dfrac{\left\vert x\right\vert
^{l/2}}{u(x)}d\mathcal{H}^{1}\right) d t .
\end{equation*}
The last inequality together with Lemma \ref{Lemma_Dopo_Fubini} with $f=1$ yield 
\begin{equation}
2\pi \left( l+2\right) \tau \leq \int_{0}^{\tau }\left( -\mu ^{\prime
}(\sigma )\sigma \right) d\sigma +\frac{1}{2\beta }\left\vert \Omega
\right\vert _{l},\text{ \ }\tau \geq v_{m}.  \label{d}
\end{equation}
Again for the solution $v$ of the symmetrized problem we have the equality
sign 
\begin{equation}
2\pi \left( l+2\right) \tau =\int_{0}^{\tau }\left( -\phi ^{\prime }(\sigma
)\sigma \right) d\sigma +\frac{1}{2\beta }\left\vert \Omega \right\vert _{l},
\text{ \ }\tau \geq v_{m}.  \label{e}
\end{equation}
From (\ref{d}) and (\ref{e}) we immediately deduce that 
\begin{equation*}
\int_{0}^{\tau }\left( -\phi ^{\prime }(\sigma )\sigma \right) d\sigma \leq
\int_{0}^{\tau }\left( -\mu ^{\prime }(\sigma )\sigma \right) d\sigma ,\text{
\ }\tau \geq v_{m}.
\end{equation*}
An integration by parts ensures that 
\begin{equation}
\mu (\tau )\leq \phi (\tau ),\text{ \ \ }\tau \geq v_{m}.  \label{mu<phi}
\end{equation}
Our claim holds true, since, clearly, $\phi (\tau )=\left\vert \Omega
\right\vert _{l}\geq \mu (\tau )$ \ for all $\tau \in \left[ 0,v_{m}\right] $
.
\end{proof}

\vspace{0.5cm}


\section{A Faber-Krahn inequality}

Consider the functional 
\begin{equation}
F:u\in H^{1}(\Omega )\rightarrow \frac{\displaystyle\int_{\Omega }\left\vert
\nabla u\right\vert ^{2}dx
+
\beta \displaystyle\int_{\partial \Omega
}u^{2}\left\vert x\right\vert ^{l/2}d\mathcal{H}^{1}}{\displaystyle 
\int_{\Omega }u^{2}\left\vert x\right\vert ^{l}dx}.  \label{F}
\end{equation}
Firstly note that $F$ is well defined on $H^{1}(\Omega ).$ 
Indeed, 
as well-known (see e.g. \cite{A}) $L^{q}(\partial \Omega )$ is compactly
embedded in $H^{1}(\Omega )$ for any $q\geq 1.$ Therefore, since $0\in
\Omega ,$ we have that $\exists C=C(\Omega ,l):$ 
\begin{equation*}
\int_{\partial \Omega }u^{2}\left\vert x\right\vert ^{l/2}d\mathcal{H}^{1}\leq
C(\Omega ,l)\int_{\partial \Omega }u^{2}d\mathcal{H}^{1}\leq C(\Omega
,l)\int_{\Omega }\left( u^{2}+\left\vert \nabla u\right\vert ^{2}\right) dx 
\text{ \ }\forall u\in H^{1}(\Omega ).
\end{equation*}
Finally let us show that there exists a constant $C=C(\Omega ,l)$ such that 
\begin{equation*}
\int_{\Omega }u^{2}\left\vert x\right\vert ^{l}dx=\left\Vert u\right\Vert
_{L^{2}(\Omega ;\left\vert x\right\vert ^{l})}^{2}\leq C\int_{\Omega }\left(
u^{2}+\left\vert \nabla u\right\vert ^{2}\right) dx\text{ \ }\forall u\in
H^{1}(\Omega ).
\end{equation*}
Since the embedding of $H^{1}(\Omega )$ in $L^{p}(\Omega )$ is compact for
any $p\geq 1$ we have 
\begin{equation}
\int_{\Omega }u^{2}\left\vert x\right\vert ^{l}dx\leq \left(
\int_{E}u_{n}^{2p}dx\right) ^{\frac{1}{p}}\left( \int_{E}\frac{1}{\left\vert
x\right\vert ^{\left\vert l\right\vert q}}dx\right) ^{\frac{1}{q}}
\label{holder}
\end{equation}
where $\dfrac{1}{p}+\dfrac{1}{q}=1$ and $p,q>1.$ Now choose $q=\widetilde{q}
\in \left( 1,\dfrac{2}{\left\vert l\right\vert }\right) ,$ so that\ $
\left\vert l\right\vert q<2,$ and $p=\widetilde{p}=\dfrac{\widetilde{q}}{ 
\widetilde{q}-1}$ in (\ref{holder}) we obtain 
\begin{equation*}
\int_{\Omega }u^{2}\left\vert x\right\vert ^{l}dx\leq \left( \int_{\Omega
}u^{2\widetilde{p}}dx\right) ^{\frac{1}{\widetilde{p}}}\left( \int_{\Omega } 
\frac{1}{\left\vert x\right\vert ^{\left\vert l\right\vert \widetilde{q}}}
dx\right) ^{\frac{1}{\widetilde{q}}}\leq C\left( \int_{\Omega }\frac{1}{
\left\vert x\right\vert ^{\left\vert l\right\vert \widetilde{q}}}dx\right)
^{ \frac{1}{\widetilde{q}}}\leq C\left\vert \Omega \right\vert ^{\frac{2+l 
\widetilde{q}}{2\widetilde{q}}}.
\end{equation*}
Incidentally note that the arguments above also show that $H^{1}(\Omega )$
is compactly embedded in $L^{2}(\Omega ;\left\vert x\right\vert ^{l}dx)$.

Now we can consider the problem 
\begin{equation*}
\inf_{w\in H^{1}(\Omega )\backslash \left\{ 0\right\} }F(u) 
=
 \inf_{w\in
H^{1}(\Omega )\backslash \left\{ 0\right\} }\frac{\displaystyle\int_{\Omega
}\left\vert \nabla w\right\vert ^{2}dx+\beta \displaystyle\int_{\partial
\Omega }w^{2}\left\vert x\right\vert ^{l/2}d\mathcal{H}^{1}}{\displaystyle 
\int_{\Omega }w^{2}\left\vert x\right\vert ^{l}dx}.
\end{equation*}
We claim that the infimum above is attained. To this aim let $\left\{
u_{n}\right\} _{n\in \mathbb{N}}$ be a minimizing sequence. Without loss of
generality we may assume that 
\begin{equation}
\int_{\Omega }u_{n}^{2}\left\vert x\right\vert ^{l}dx=1\text{ \ }\forall
n\in \mathbb{N}.  \label{1}
\end{equation}
Clearly such a sequence is bounded in $H^{1}(\Omega ).$ Therefore, up to a
subsequence, we have that there exists $u\in H^{1}(\Omega )$ such that 
\begin{equation}
\left\{ 
\begin{array}{c}
u_{n}\rightarrow u\text{ \ a.e. in }\Omega \\ 
\\ 
u_{n}\rightarrow u\text{ in }L^{p}(\Omega )\text{ \ }\forall p\geq 1 \\ 
\\ 
u_{n}\rightarrow u\text{ weakly in }H^{1}(\Omega ) \\ 
\\ 
\displaystyle\int_{\partial \Omega }u_{n}^{2}\left\vert x\right\vert ^{l/2}d\mathcal{H}
^{1}\rightarrow \int_{\partial \Omega }u^{2}\left\vert x\right\vert ^{l/2}d 
\mathcal{H}^{1}
\end{array}
\right.  \label{trace_emb}
\end{equation}
with 
\begin{equation*}
\int_{\Omega }u^{2}\left\vert x\right\vert ^{l}dx=1.
\end{equation*}
Hence, using (\ref{1}), (\ref{trace_emb}) and the weak lower semicontinuity
of the $L^{2}-$norm of the gradient we obtain our claim 
\begin{equation*}
\liminf_{n\rightarrow \infty }F(u_{n})=F(u)=\min_{w\in H^{1}(\Omega )\backslash
\left\{ 0\right\} }F(w).
\end{equation*}
Let $\lambda _{1,l}(\Omega )$ denote the first eigenvalue of the problem 
\begin{equation}
\left\{ 
\begin{array}{cc}
-\Delta u=\lambda (\Omega )\left\vert x\right\vert ^{l}u & \text{in \ }\Omega
\\ 
&  \\ 
\dfrac{\partial u}{\partial \nu }+\beta \left\vert x\right\vert ^{l/2}u=0 & 
\text{on \ }\partial \Omega .
\end{array}
\right.
\end{equation}
By the consideration above $\lambda _{1,l}(\Omega )$ has the following
variational characterization

\begin{equation*}
\lambda _{1,l}(\Omega )=\min_{w\in H^{1}(\Omega )\backslash \left\{
0\right\} }\frac{\displaystyle\int_{\Omega }\left\vert \nabla w\right\vert
^{2}dx
+
\beta \displaystyle\int_{\partial \Omega }w^{2}\left\vert
x\right\vert ^{l/2}d\mathcal{H}^{1}}{\displaystyle\int_{\Omega
}w^{2}\left\vert x\right\vert ^{l}dx}.
\end{equation*}
Now we are in position to prove the weighted Faber-Krahn inequality (\ref
{F_K}), see also \cite{Ke} and \cite{ANT}.

\vspace{0.5cm}

\begin{proof}[Proof of Theorem \protect\ref{FK}.]
Let $u_{1}$ be an eigenfunction corresponding to $\lambda _{1,l}(\Omega ),$
that is 
\begin{equation}
\left\{ 
\begin{array}{cc}
-\Delta u_{1}=\lambda _{1,l}(\Omega )\left\vert x\right\vert ^{l}u_{1} & 
\text{in \ }\Omega  \\ 
&  \\ 
\dfrac{\partial u_{1}}{\partial \nu }+\beta \left\vert x\right\vert
^{l/2}u_{1}=0 & \text{on \ }\partial \Omega .
\end{array}
\right. 
\end{equation}
Denoting by $z$ the solution to 
\begin{equation}
\left\{ 
\begin{array}{cc}
-\Delta z=\lambda _{1,l}(\Omega )\left\vert x\right\vert ^{l}u_{1}^{\sharp }
& \text{in \ }\Omega ^{\sharp } \\ 
&  \\ 
\dfrac{\partial z}{\partial \nu }+\beta 
(r^{\sharp })^{l/2}z=0 & 
\text{on \ }\partial \Omega ^{\sharp }.
\end{array}
\right.   \label{Eq_z}
\end{equation}
Inequality (\ref{L^2}) gives 
\begin{equation*}
\int_{\Omega }u_{1}^{2}\left\vert x\right\vert ^{l}dx=\int_{\Omega ^{\sharp
}}\left( u_{1}^{\sharp }\right) ^{2}\left\vert x\right\vert ^{l}dx\leq
\int_{\Omega ^{\sharp }}z^{2}\left\vert x\right\vert ^{l}dx,
\end{equation*}
which, together with the Cauchy-Schwarz inequality, implies 
\begin{equation}
\label{CS}
\int_{\Omega ^{\sharp }}u_{1}^{\sharp }z\left\vert x\right\vert ^{l}dx\leq
\left( \int_{\Omega ^{\sharp }}\left( u_{1}^{\sharp }\right) ^{2}\left\vert
x\right\vert ^{l}dx\right) ^{\frac{1}{2}}\left( \int_{\Omega ^{\sharp
}}z^{2}\left\vert x\right\vert ^{l}dx\right) ^{\frac{1}{2}}\leq \int_{\Omega
^{\sharp }}z^{2}\left\vert x\right\vert ^{l}dx.
\end{equation}
Multiplying equation (\ref{Eq_z}) by $z$,  integrating and taking into account
of (\ref{CS}) we finally obtain 
\begin{eqnarray*}
\lambda _{1,l}(\Omega ) 
&=&
\frac{\displaystyle\int_{\Omega ^{\sharp
}}\left\vert \nabla z\right\vert ^{2}dx+\beta \displaystyle\int_{\partial
\Omega ^{\sharp }}z^{2}\left\vert x\right\vert ^{l/2}d\mathcal{H}^{1}}
{\displaystyle\int_{\Omega ^{\sharp }}  u_{1}^{\sharp }   z  \left\vert x\right\vert ^{l}dx} \\
&\geq &
\frac{\displaystyle\int_{\Omega ^{\sharp }}\left\vert \nabla
z\right\vert ^{2}dx+\beta \displaystyle\int_{\partial \Omega ^{\sharp
}}z^{2}\left\vert x\right\vert ^{l/2}d\mathcal{H}^{1}}
{\displaystyle
\int_{\Omega ^{\sharp }}z^{2}\left\vert x\right\vert ^{l}dx}\geq \lambda
_{1,l}(\Omega ^{\sharp }).
\end{eqnarray*}
\end{proof}

\end{document}